%% file: original.tex
\documentclass[a4paper,12pt]{amsart}

\usepackage{enumerate,layout,amssymb,epsf,mathrsfs}

\newtheorem{theorem}{Theorem}[section]

\newtheorem{proposition}[theorem]{Proposition}

\newtheorem{lemma}[theorem]{Lemma}

\theoremstyle{definition}

\newtheorem{definition}[theorem]{Definition}

\newtheorem{remark}[theorem]{Remark}

\newcommand{\Z}{{\mathbb Z}}
\newcommand{\N}{{\mathbb N}}
\newcommand{\Q}{{\mathbb Q}}

\numberwithin{equation}{section}

\begin{document}

\title[Comparability of clopen sets]{Comparability of clopen sets in a 
zero-dimensional dynamical system}

\author[H. Yuasa]{Hisatoshi Yuasa}

\address{
17-23-203 Idanakano-cho, Nakahara-ku, Kawasaki Kanagawa 211-0034, JAPAN.}

\email{hisatoshi\_yuasa@ybb.ne.jp}

\keywords{unique ergodicity, totally ordered group.}

\subjclass[2000]{ Primary 37B05}

\maketitle

\input{body.tex}

\end{document}

%% file: body.tex
\begin{abstract} 
Let $\varphi$ be a homeomorphism on a totally disconnected, compact metric space 
$X$. Then, the following are equivalent:
\begin{enumerate}[(i)]
\item
$\varphi$ is uniquely ergodic;
\item
any clopen subsets of $X$ are comparable with respect to a certain binary 
relation;
\item
an ordered group, which is a quotient of the group of 
integer\nobreakdash-valued continuous functions modulo infinitesimals, is 
totally ordered.
\end{enumerate}
\end{abstract}

\section{Introduction}\label{intro}

Let $\varphi$ be a homeomorphism on a totally disconnected, 
compact metric space $X$. Let $M_\varphi$ denote the set of $\varphi$-invariant 
probability measures. 
For clopen sets $A,B \subset X$, we write $A \ge B$ 
either if $\mu(A) > \mu(B)$ for all $\mu \in M_\varphi$, or if $\mu(A)=\mu(B)$ 
for all $\mu \in M_\varphi$. This relation $\ge$ does not necessarily hold between 
given clopen sets (Remark~\ref{incomp}). 
If $\varphi$ is minimal, then $A \ge B$ induces an 
embedding of $B$ into $A$ via finite or countable Hopf-equivalence \cite{GW}. 
The embedding plays significant roles in analyses of orbit 
structures of Cantor minimal systems \cite{GPS2,GW,HKY} and also in those for 
locally compact Cantor minimal systems \cite{Matui}. We refer the reader to 
\cite{Y1,Y5} for other facts concerning Hopf-equivalence. 

Another important object in analyses of the orbit structures is ordered group. 
Let $G_\varphi$ denote the quotient group of the abelian group $C(X,\Z)$ of 
integer-valued continuous functions on $X$ by a subgroup:
\[
Z_\varphi=\{f \in C(X,\Z)|\int_X f d \mu = 0 \textrm{ for all } \mu \in 
M_\varphi\}.
\]
Let $G_\varphi^+=\{[f] \in G_\varphi| f \ge 0\}$, where $[f]$ is the equivalence 
class of $f \in C(X,\Z)$. If $\varphi$ is minimal, then 
the ordered group $(G_\varphi,G_\varphi^+)$ with the canonical order unit is a 
complete invariant for orbit equivalence \cite{GPS}.

If $\varphi$ is uniquely ergodic, then any clopen subsets of $X$ are 
comparable. This fact may lead us to have a question whether a 
non\nobreakdash-uniquely ergodic 
system always has incomparable clopen sets, or not. The goal of this paper is 
to give an affirmative answer to this question: 
\begin{theorem}\label{essential}
The following are equivalent{\rm :}
\begin{enumerate}[{\rm (i)}]
\item\label{UE}
$\varphi$ is uniquely ergodic{\rm ;}
\item\label{COMP}
any two clopen subsets of $X$ are comparable{\rm ;}
\item\label{TOTALLY}
the ordered group $(G_\varphi,G_\varphi^+)$ is totally ordered.
\end{enumerate}
\end{theorem}

\paragraph{{\it Acknowledgements}} Although the previous version of this paper was 
rejected by a journal, the referee gave the author useful comments, 
which made a proof of the theorem much shorter. The author appreciates it.

\section{Preliminaries}

We freely use terminology concerning (partially) ordered groups or 
dimension groups; see \cite{DHS,Effros,GPS} for example. 
Let $K^0(X,\varphi)$ denote a quotient group $C(X,\Z)/B_\varphi$, where 
$B_\varphi=\{f \circ \varphi -f | f \in C(X,\Z)\}$. 
Put $K^0(X,\varphi)^+=\{[f] \in K^0(X,\varphi)|f \ge 0\}$. 
If any point in $X$ is chain recurrent for $\varphi$, then 
$(K^0(X,\varphi),K^0(X,\varphi)^+)$ becomes an ordered group \cite{BH}. 
This fact is proved also by \cite{Y5} in connection with finite 
Hopf\nobreakdash-equivalence. If $\varphi$ is minimal (resp.\ almost 
minimal), then $(K^0(X,\varphi),K^0(X,\varphi)^+)$ becomes a simple 
(resp.\ almost simple) dimension group \cite{HPS} 
(resp.\ \cite{Danilenko}). In each of these cases, 
$(K^0(X,\varphi),K^0(X,\varphi)^+)$ with the canonical order unit $[\chi_X]$ 
is a complete invariant for strong orbit equivalence \cite{GPS,Danilenko}, 
where $\chi_X$ is the characteristic function of $X$. 

Suppose $\varphi$ has a unique minimal set. By \cite[Theorem~1.1]{HPS}, any 
point is chain recurrent for $\varphi$. Given $\mu \in M_\varphi$, define a 
state $\tau_\mu$ on $(K^0(X,\varphi),[\chi_X])$ by for $f \in C(X,\Z)$, 
\begin{equation*}\label{correspondence}
\tau_\mu([f])=\int_X f d \mu.
\end{equation*}
The map $\mu \mapsto \tau_\mu$ is a bijection between $M_\varphi$ and the set 
of states on $(K^0(X,\varphi),[\chi_X])$; see for details \cite[Theorem~5.5]{HPS}. 
\begin{proposition}\label{misc}
$(G_\varphi, G_\varphi^+)$ is an ordered group. 
\end{proposition}
\begin{proof}
Suppose $[f] \in G_\varphi^+ \cap (-G_\varphi^+)$ with $f \in C(X,\Z)$. 
There are nonnegative $g_1,g_2 \in C(X,\Z)$ such that 
$f-g_1,f+g_2 \in Z_\varphi$. Since for all $\mu \in M_\varphi$,
\[
0=\int_X (g_1+g_2)d \mu \ge \int_X g_1 d \mu \ge 0, 
\]
we obtain $[f]=[g_1]=0$, i.e.\ $G_\varphi^+ \cap (-G_\varphi^+)=\{0\}$. Other 
requirements for $(G_\varphi,G_\varphi^+)$ to be an ordered group are 
readily verified. 
\end{proof}
\begin{definition}
Clopen subsets $A$ and $B$ of $X$ are said to be {\em countably Hopf-equivalent} 
if there exist $\{n_i \in \Z|i \in \Z^+\}$ and disjoint unions 
\[
A=\bigcup_{i \in \N}A_i \cup \{x_0\} \text{ and } 
B=\bigcup_{i \in \N}B_i \cup \{y_0\}
\]
into nonempty clopen sets $A_i,B_i$ and singletons $\{x_0\}, \{y_0\}$ such that 
\begin{enumerate}
\item
$\varphi^{n_0}(x_0)=y_0$ and $\varphi^{n_i}(A_i)=B_i$ for every $i \in \N;$
\item
the map $\alpha:A \to B$ defined by 
\[
\alpha(x)=
\begin{cases}
\varphi^{n_i}(x) & \text{ if } x \in A_i \text{ and } i \in \N; \\
y_0 & \text{ if } x=x_0
\end{cases}
\]
is a homeomorphism. 
\end{enumerate}
We shall refer to $\alpha$ as a {\em countable equivalence map} from $A$ onto 
$B$.  
\end{definition}

\begin{lemma}\label{three}
Suppose $\varphi$ is minimal. Let $A,B \subset X$ be clopen. Then, the 
following are equivalent$:$
\begin{enumerate}[{\rm (1)}]
\item\label{the_relation}
$A \ge B;$
\item\label{CEM}
there is a countable equivalence map from $B$ into $A;$
\item\label{in_D}
$[\chi_A]-[\chi_B] \in D_\varphi:= 
\{[\chi_C] \in G_\varphi|C \subset X \text{ is clopen.}\}$.
\end{enumerate}
\end{lemma}

\begin{proof}
By \cite[Proposition~2.6]{GW}, $\eqref{the_relation}$ is equivalent to 
$\eqref{CEM}$. If $\alpha:B \to \alpha(B) \subset A$ is a 
countable equivalence map, then 
\[
[\chi_A]-[\chi_B]=[\chi_A]-[\chi_{\alpha(B)}]=[\chi_{A \setminus \alpha(B)}] 
\in D_\varphi.
\]
Hence, $\eqref{CEM}$ implies $\eqref{in_D}$. If $[\chi_A]-[\chi_B]=[\chi_C]$ for 
some clopen set $C \subset X$, then $\int (\chi_A-\chi_B)d \mu=\mu(C) \ge 0$ 
for all $\mu \in M_\varphi$. Since given a clopen set $C \subset X$, either 
$\mu(C)=0$ for all $\mu \in M_\varphi$, or $\mu(C)>0$ for all 
$\mu \in M_\varphi$, $\eqref{in_D}$ implies $\eqref{the_relation}$. 
This completes the proof. 
\end{proof}

\begin{remark}\label{incomp}
One can find a homeomorphism having incomparable clopen sets. 
Let $(X,\varphi)$ be a Cantor minimal system such that 
$K^0(X,\varphi)$ is order isomorphic to $\Q^2$  with the 
strict ordering by an isomorphism $\iota$ mapping the canonical order unit 
$[\chi_X]$ to $(1,1)$; see for details \cite{GPS,DHS,HPS}. 
The homeomorphism $\varphi$ has exactly two ergodic probability measures 
corresponding to states $\tau_i:\Q^2 \to \Q$ ($i=1,2$) which are the projections 
to the $i$\nobreakdash-th coordinate. 
By \cite[Lemma~2.4]{GW}, there exist clopen sets $C,D \subset X$ such that 
$\iota([\chi_C])=(1/2,1/3)$ and $\iota([\chi_D])=(1/2,2/3)$, which are 
incomparable. 
\end{remark}

\section{A proof of the theorem}
 
$\eqref{COMP} \Rightarrow \eqref{TOTALLY}$: 
We first show that $\varphi$ has a 
unique minimal set on which any $\mu \in M_\varphi$ is supported. 
Let $Y$ be a minimal set. Suppose $\mu \in M_\varphi$ is supported on $Y$. 
Suppose $\nu \in M_\varphi$ is different from $\mu$. Assume $\nu(A)>0$ 
for a clopen set $A \subset X \setminus Y$. Define $\nu^\prime \in M_\varphi$ 
by $\nu^\prime(U)=\nu(U \setminus Y)/\nu(X \setminus Y)$ for a measurable set 
$U$. By regularity, there exists a clopen set $B$ containing $Y$ such that 
$\nu^\prime(B) < \nu^\prime(A)$. However, $\mu(B)=1$ and $\mu(A)=0$. 
This contradicts~\eqref{COMP} . 

In the remainder of this proof, we tacitly use Lemma~\ref{three}. The fact 
proved in the preceding paragraph allows us to assume the minimality of $\varphi$. 
Given $a \in G_\varphi$, choose $\{a_i,b_j \in D_\varphi \setminus 
\{0\} | 1 \le i \le n, 1 \le j \le m\}$ so that 
$a=a_1+a_2+\dots +a_n-b_1-b_2- \dots -b_m$.
 
The following procedure consisting of steps determines $a \ge 0$ or $a \le 0$. 
\paragraph{{\bf Step~1}}
If $\sum_{i=1}^na_i-b_1 \le 0$, then $a \le 0$, and the procedure ends. 
Otherwise, there is $k_1$ for which 
$c_{k_1} := \sum_{i=1}^{k_1}a_i-b_1 \in D_\varphi \setminus \{0\}$ and 
$a =c_{k_1}+a_{k_1+1}+\dots + a_n-b_2-b_3- \dots -b_m$. By this operation, 
the number of terms $b_i$ decreases by one. We may 
write $a=a_{k_1}+a_{k_1+1}+\dots + a_n-b_2-b_3- \dots -b_m$. 
\paragraph{{\bf Step~2}} If $\sum_{i=k_1}^na_i - b_2 \le 0$, then $a \le 0$, and 
the procedure ends. Otherwise, there is $k_2 \ge k_1$ for which 
$
c_{k_2} :=\sum_{i=k_1}^{k_2}a_i-b_2 \in D_\varphi \setminus \{0\}$ and 
$a =c_{k_2}+a_{k_2+1}+ \dots +a_n-b_3-b_4-\dots -b_m$. 
By this operation, the number of terms $b_i$ decreases by one. We may 
write $a=a_{k_2}+a_{k_2+1}+\dots + a_n -b_3-b_4- \dots -b_m$. 

Now, it is clear how we should execute each step. The procedure necessarily 
ends by Step $m$. We obtain $a \ge 0$ exactly when the procedure ends at Step 
$m$.

$\eqref{TOTALLY} \Rightarrow \eqref{UE}$: 
Assume the existence of a clopen set $A \subset X$ such that 
$c_2:=\inf_{\mu \in M_\varphi}\int \chi_A d\mu < \sup_{\mu \in M_\varphi} 
\int \chi_A d\mu =:c_1$. 
Let $\mu_1,\mu_2 \in M_\varphi$ be so that $c_i=\int \chi_A d \mu_i$. 
Take $m,n \in \N$ so that $c_2<n/m<c_1$. Then, 
$\int (m\chi_A-n)d \mu_1>0$ and $\int (m\chi_A-n)d \mu_2<0$. This contradicts 
\eqref{TOTALLY}. This completes the proof of the theorem.

\begin{remark}
The proof of $\eqref{COMP} \Rightarrow \eqref{TOTALLY}$ is based on an idea 
implied in \cite[Subsection~5.4]{El}.
\end{remark}